\documentclass[11pt]{article}

\usepackage{amsmath, amssymb, amsthm} 
\usepackage{stmaryrd} 
\usepackage{thmtools} 
\usepackage{xspace}
\allowdisplaybreaks
\expandafter\let\expandafter\oldproof\csname\string\proof\endcsname
\let\oldendproof\endproof
\renewenvironment{proof}[1][\proofname]{%
	\oldproof[\bf #1]%
}{\oldendproof}

\allowdisplaybreaks

\theoremstyle{plain}
\newtheorem{theorem}{Theorem}[section]
\newtheorem{lemma}[theorem]{Lemma}
\newtheorem{claim}[theorem]{Claim}
\newtheorem{proposition}[theorem]{Proposition}
\newtheorem{observation}[theorem]{Observation}
\newtheorem{corollary}[theorem]{Corollary}

\newtheorem{remark}[theorem]{Remark}
\newtheorem{definition}[theorem]{Definition}

\newcommand{\C}{\mathcal C}
\newcommand{\F}{\mathcal F}
\newcommand{\R}{\mathcal R}

\newcommand{\B}{\mathcal B}
\newcommand{\G}{\mathcal G}
\renewcommand{\P}{\mathcal P}

\usepackage{mathrsfs}
\usepackage{stackrel}

\newcommand{\sG}{\mathscr{G}}
\newcommand{\HH}{\mathcal{H}}
\newcommand{\sS}{\mathcal{S}}
\newcommand{\sC}{\mathscr{C}}

\newcommand{\mand}{\; \text{and}\;}

\DeclareMathOperator{\Pro}{\mathbb{P}}
\DeclareMathOperator{\VAR}{{\bf Var}}
\DeclareMathOperator{\COV}{{\bf Cov}}
\DeclareMathOperator{\Ex}{\mathbb{E}}

\newcommand{\floor}[1]{\left\lfloor #1 \right\rfloor}
\newcommand{\ceil}[1]{\left\lceil #1 \right\rceil}
\newcommand{\Gnp}{\mathbb{G}(n,p)}
\def\rainbow{\stackrel{\mathrm{rbw}}{\longrightarrow}}
\def\notrainbow{\stackrel{\mathrm{rbw}}{\longarrownot\longrightarrow}}
\newcommand{\extends}[2]{\stackrel[\scriptscriptstyle #1]{\hookrightarrow}{} #2}
\makeatletter
\def\moverlay{\mathpalette\mov@rlay}
\def\mov@rlay#1#2{\leavevmode\vtop{%
   \baselineskip\z@skip \lineskiplimit-\maxdimen
   \ialign{\hfil$\m@th#1##$\hfil\cr#2\crcr}}}
\newcommand{\charfusion}[3][\mathord]{
    #1{\ifx#1\mathop\vphantom{#2}\fi
        \mathpalette\mov@rlay{#2\cr#3}
      }
    \ifx#1\mathop\expandafter\displaylimits\fi}
\makeatother

\newcommand{\discup}{\charfusion[\mathbin]{\cup}{\cdot}}

\let\eps=\varepsilon
\let\theta=\vartheta
\let\rho=\varrho
\let\phi=\varphi

\usepackage{framed}

\usepackage{enumitem}
\usepackage[bookmarks = false]{hyperref}

\makeatletter
\renewcommand*{\eqref}[1]{%
  \hyperref[{#1}]{\textup{\tagform@{\ref*{#1}}}}%
}
\makeatother

\usepackage{tikz}
\usetikzlibrary{calc}
\usetikzlibrary{patterns}
\usetikzlibrary{arrows,decorations.pathmorphing,backgrounds,positioning,fit,petri}
\usetikzlibrary{decorations.text}
\usetikzlibrary{matrix}

\definecolor{myMaroon}{HTML}{720E0E}
\definecolor{myBlue}{HTML}{1A5276}
\definecolor{myDarkerBlue}{HTML}{154360}
\definecolor{my_blue}{HTML}{4A5DE2}
\definecolor{my_purple}{HTML}{9013FE}
\definecolor{my_red}{HTML}{D0021B}
\definecolor{my_cyan}{HTML}{2CA78B}
\definecolor{my_green}{HTML}{417505}
\definecolor{my_yellow}{HTML}{F5A623}

\definecolor{1}{HTML}{000000}
\definecolor{2}{HTML}{000000}
\definecolor{6}{HTML}{000000}
\definecolor{4}{HTML}{000000}
\definecolor{7}{HTML}{000000}
\definecolor{3}{HTML}{000000}
\definecolor{5}{HTML}{000000}

\tikzset{colour1/.style={
        color = 1,
    }
}
\tikzset{colour1/.style={
        color = 1,
    }
}\tikzset{colour2/.style={
        color = 2,
    }
}\tikzset{colour3/.style={
        color = 3,
    }
}\tikzset{colour4/.style={
        color = 4,
        densely dotted
    }
}\tikzset{colour5/.style={
        color = 5,
        densely dashed
    }
}\tikzset{colour6/.style={
        color = 6,
        densely dashdotted
    }
}\tikzset{colour7/.style={
        color = 7,
        dash pattern=on 3mm off 1mm
    }
}

\makeatletter
    \def\@fnsymbol#1{\ensuremath{\ifcase#1\or *\or \mathsection\or \ddagger\or
       \dagger\or \mathparagraph\or \|\or **\or \dagger\dagger
       \or \ddagger\ddagger \else\@ctrerr\fi}}
\makeatother

\title{Large rainbow cliques in randomly perturbed dense graphs}

\author{
	Elad Aigner-Horev \thanks{Department of Computer Science, Ariel University, Ariel 40700, Israel. Email: {\tt horev@ariel.ac.il}.} 
	\and 
	Oran Danon \thanks{Department of Computer Science, Ariel University, Ariel 40700, Israel. Email: {\tt oran.danon@msmail.ariel.ac.il}.}
	\and 
	Dan Hefetz \thanks{Department of Computer Science, Ariel University, Ariel 40700, Israel. Email: {\tt danhe@ariel.ac.il}. Research supported by ISF grant 822/18.}
	\and 
	Shoham Letzter \thanks{Department of Mathematics and Mathematical Statistics, University of Cambridge, Wilberforce Road, Cambridge CB3 0WB, UK. Email: {\tt s.letzter@dpmms.cam.ac.uk}. Research supported by the Royal Society.}
}

\begin{document}

\clearpage\maketitle
\thispagestyle{empty}

\begin{abstract}
	For two graphs $G$ and $H$, write $G \rainbow H$ if $G$ has the property that every {\sl proper} colouring of its edges yields a {\sl rainbow} copy of $H$.
	We study the thresholds for such so-called {\sl anti-Ramsey} properties in randomly perturbed dense graphs, which are unions of the form $G \cup \mathbb{G}(n,p)$, where $G$ is an $n$-vertex graph with edge-density at least $d$, and $d$ is a constant that does not depend on $n$. 
	 
	Our results in this paper, combined with our results in a companion paper, determine the threshold for the property $G \cup \Gnp \rainbow K_s$ for every $s$. In this paper, we show that for $s \geq 9$ the threshold is $n^{-1/m_2(K_{\ceil{s/2}})}$; in fact, our $1$-statement is a supersaturation result. This turns out to (almost) be the threshold for $s=8$ as well, but for every $4 \leq s \leq 7$, the threshold is lower; see our companion paper for more details.

Also in this paper, we determine that the threshold for the property $G \cup \Gnp \rainbow C_{2\ell - 1}$ is $n^{-2}$ for every $\ell \geq 2$; in particular, the threshold does not depend on the length of the cycle $C_{2\ell - 1}$. For even cycles, and in fact any fixed bipartite graph, no random edges are needed at all; that is $G \rainbow H$ always holds, whenever $G$ is as above and $H$ is bipartite.

\end{abstract}

\section{Introduction}

	A {\sl random perturbation} of a fixed $n$-vertex graph $G$, denoted by $G \cup \mathbb{G}(n,p)$, is a distribution over the supergraphs of $G$. The elements of such a distribution are generated via the addition of randomly sampled edges to $G$. These random edges are taken from the binomial random graph with edge-probability $p$, namely $\mathbb{G}(n,p)$. The fixed graph $G$ being {\sl perturbed} or {\sl augmented} in this manner is referred to as the {\em seed} of the {\em perturbation} (or {\em augmentation}) $G \cup \mathbb{G}(n,p)$. 

	The above model of randomly perturbed graphs was introduced by Bohman, Frieze, and Martin~\cite{BFM03}, who allowed the seed $G$ to range over the family of $n$-vertex graphs with minimum degree at least $\delta n$, which we denote here by $\G_{\delta, n}$. In particular, they discovered the phenomenon that for every $\delta > 0$, there exists a constant $C(\delta) > 0$ such that $G \cup \mathbb{G}(n,p)$ asymptotically almost surely (henceforth a.a.s.\ for brevity) admits a Hamilton cycle, whenever $p: = p(n) \geq C(\delta)/n$ and $G \in \G_{\delta,n}$. Note that the value of $p$ attained by their result is smaller by a logarithmic factor than that required for the emergence of Hamilton cycles in $\mathbb{G}(n,p)$. That is, while $G$ itself might not be Hamiltonian, making it Hamiltonian requires far fewer random edges than the number of random edges which typically form a Hamilton cycle by themselves. The notation $\G_{\delta,n} \cup \mathbb{G}(n,p)$ then suggests itself to mean the collection of perturbations arising from the members of $\G_{\delta,n}$ for a prescribed $\delta > 0$. 

	Several strands of results regarding the properties of randomly perturbed (hyper)graphs can be found in the literature. One prominent such strand can be seen as an extension of the results found in~\cite{BFM03} pertaining to the Hamiltonicity of $\G_{\delta,n} \cup \mathbb{G}(n,p)$. Indeed, the emergence of various spanning configurations in randomly perturbed (hyper)graphs was studied, for example, in~\cite{BTW17, BHKM18, BHKMPP18, BMPP18, DRRS18, HZ18, KKS16, KKS17, MM18}.

	Another prominent line of research regarding random perturbations concerns Ramsey properties of $\sG_{d,n} \cup \mathbb{G}(n,p)$, where here $\sG_{d,n}$ stands for the family of $n$-vertex graphs with edge-density at least $d > 0$, and $d$ is a constant. This strand stems from the work of Krivelevich, Sudakov, and Tetali~\cite{KST}. This line of research is heavily influenced by the now fairly mature body of results regarding the thresholds of various Ramsey properties in random graphs. 

	The study of Ramsey properties in random graphs was initiated by \L{}uczak, Ruci\'nski, and Voigt~\cite{LRV92}. The so-called {\sl symmetric edge Ramsey problem} for random graphs was settled completely by R\"odl and Ruci\'nski in a series of papers~\cite{RR93,RR94,RR95} that collectively established the so-called {\sl symmetric Random Ramsey theorem}. The best known consequence of this theorem is that for every integer $r \geq 2$ and every graph $H$ containing a cycle, there exist constants $c:= c(r,H)$ and $C:=C(r,H)$ such that 
	\begin{equation}\label{eq:sym-random-Ramsey}
	\lim_{n \to \infty} \Pro[\mathbb{G}(n,p) \to (H)_r] = 
	\begin{cases}
	0, & \;\text{if $p \leq c n^{-1/m_2(H)}$},\\
	1, & \;\text{if $p \geq C n^{-1/m_2(H)}$}.
	\end{cases}
	\end{equation}
	Here, $G \rightarrow (H)_r$ is the classical arrow notation used in Ramsey theory to denote that the graph $G$ has the property that every $r$-edge-colouring of $G$ admits a monochromatic copy of $H$. Having the same configuration $H$ sought in every colour, lends this type of results the title of being {\sl symmetric}. The parameter $m_2(H)$ is the so-called {\em maximum $2$-density} of $H$ given by 
	$$
	m_2(H) := \max \left \{\frac{e(F)-1}{v(F)-2} : F \subseteq H, e(F) \geq 2 \right\}.
	$$
	The symmetric random Ramsey theorem determines the threshold for the property $\mathbb{G}(n,p) \to (H)_r$ for all $H$, i.e., also in the case that $H$ does not contain a cycle; see, e.g.~\cite[Theorem~1]{GNPSSH17} for a complete and accurate formulation of this result.
	An alternative short proof of the $1$-statement appearing 
	in~\eqref{eq:sym-random-Ramsey} was recently provided by Nenadov and Steger~\cite{NS16} utilising the so-called {\sl container method}~\cite{BMS15,ST15}.  
	 
	As noted above, Krivelevich, Sudakov, and Tetali~\cite{KST} were the first to study Ramsey properties of random perturbations. In particular, they proved that for every real $d > 0$, integer $t \geq 3$, and graph $G \in \sG_{d,n}$, the perturbation $G \cup \mathbb{G}(n,p)$ a.a.s.\ satisfies the property $G \cup \mathbb{G}(n,p) \rightarrow (K_3, K_t)$, whenever $p:= p(n) = \omega(n^{-2/(t-1)})$; moreover, this bound on $p$ is asymptotically best possible. Here, the notation $G \rightarrow (H_1, \ldots, H_r)$ is used to denote that $G$ has the {\sl asymmetric} Ramsey property asserting that any $r$-edge-colouring of $G$ admits a colour $i \in [r]$ such that $H_i$ appears with all its edges assigned the colour $i$. 

	Recently, the aforementioned result of Krivelevich, Sudakov, and Tetali~\cite{KST} has been significantly extended by Das and Treglown~\cite{DT19} and also by Powierski~\cite{Powierski19}. In particular, there is now a significant body of results pertaining to the property $G \cup \mathbb{G}(n,p) \rightarrow (K_r,K_s)$ for any pair of integers $r,s \geq 3$, whenever $G \in \sG_{d,n}$ for constant $d > 0$. Further in this direction, the work of Das, Morris, and Treglown~\cite{DMT19} extends the results of Kreuter~\cite{K96} pertaining to {\sl vertex Ramsey} properties of random graphs into the perturbed model. 

	Krivelevich, Sudakov, and Tetali~\cite{KST} studied additional properties of $\sG_{d,n} \cup \mathbb{G}(n,p)$. In particular, they studied the so-called {\sl containment} problem of small prescribed graphs in such perturbations. The emergence of complete graphs of fixed size in this model was studied earlier by Bohman, Frieze, Krivelevich and Martin~\cite{BFKM04} who determined thresholds for their emergence in this model. 

	Sudakov and Vondr\'ak~\cite{SV08} studied the non-$2$-colourability of randomly perturbed dense hypergraphs. Furthermore, in the arithmetic-Ramsey setting, the first author and Person~\cite{AHP} established an (asymptotically) optimal Schur-type theorem for randomly perturbed dense sets of integers. 

	Problems concerning the emergence of non-monochromatic configurations in every (sensible) edge-colouring of a given graph are collectively referred to as {\em Anti-Ramsey} problems. Here, one encounters a great diversity of variants; further details can be found in the excellent survey~\cite{FMO10} and the many references therein. 

	An edge-colouring $\psi$ of a graph $G$ is said to be $b$-{\em bounded} if no colour is used on more than $b$ edges. It is said to be {\em locally-$b$-bounded} if every colour appears at most $b$ times at every vertex. In particular, locally-$1$-bounded colourings are the traditional {\em proper} colourings. 
	A subgraph $H \subseteq G$ is said to be {\em rainbow} with respect to an edge colouring $\psi$, if any two of its edges are assigned different colours under $\psi$, that is, if $|\psi(H)| := |\psi(E(H))| = e(H)$, where $\psi(E(H)) := \{\psi(e) : e \in E(H)\}$. We write $G \rainbow H$, if $G$ has the property that every proper colouring of its edges admits a rainbow copy of $H$. 

	For a fairly complete overview regarding the emergence of small fixed rainbow configurations in random graphs with respect to every $b$-bounded colouring, see the work of Bohman, Frieze, Pikhurko, and Smyth~\cite{BFPS10} and references therein. 
	The first to consider the emergence of small fixed rainbow configurations in random graphs with respect to proper colourings were R\"odl and Tuza~\cite{RT92}. In a response to a question of Spencer (see, \cite[page 19]{Erdos79}), R\"odl and Tuza studied the emergence of rainbow cycles of fixed length. 
	
	The systematic study of the emergence of general rainbow fixed graphs in random graphs with respect to proper colourings was initiated by Kohayakawa, Kostadinidis and Mota~\cite{KKM14, KKM18}. 
	They~\cite{KKM14} proved that for every graph $H$, there exists a constant $C >0$ such that $\mathbb{G}(n,p) \rainbow H$, whenever $p \geq Cn^{-1/m_2(H)}$. 
	Nenadov, Person, \v{S}kori\'{c}, and Steger~\cite{NPSS17} proved, amongst other things, that for $H \cong C_\ell$ with $\ell \geq 7$, and for $H \cong K_r$ with $ r\geq 19$, $n^{-1/m_2(H)}$ is, in fact, the threshold for the property $\mathbb{G}(n,p) \rainbow H$.  
	Barros, Cavalar, Mota, and Parczyk~\cite{BCMP19} extended the result of~\cite{NPSS17} for cycles, proving that the threshold for the property $\mathbb{G}(n,p) \rainbow C_\ell$ remains $n^{-1/m_2(C_\ell)}$ also when $\ell \geq 5$. Kohayakawa, Mota, Parczyk, and Schnitzer~\cite{KMPS18} extended the result of~\cite{NPSS17} for complete graphs, proving that the threshold of $\mathbb{G}(n,p) \rainbow K_r$ remains $n^{-1/m_2(K_r)}$ also when $r \geq 5$.  

	For $C_4$ and $K_4$ the situation is different. The threshold for the property $\mathbb{G}(n,p) \rainbow C_4$ is $n^{-3/4} = o \left(n^{-1/m_2(C_4)} \right)$, as proved by Mota~\cite{Mota17}. For the property $\mathbb{G}(n,p) \rainbow K_4$, the threshold is $n^{-7/15} = o \left(n^{-1/m_2(K_4)} \right)$ as proved by Kohayakawa, Mota, Parczyk, and Schnitzer~\cite{KMPS18}. More generally, Kohayakawa, Kostadinidis and Mota~\cite{KKM18} proved that there are infinitely many graphs $H$ for which the threshold for the property $\mathbb{G}(n,p) \rainbow H$ is significantly smaller than $n^{-1/m_2(H)}$. 

	Note that every properly-coloured triangle is rainbow, so the threshold for $\mathbb{G}(n,p)$ containing a rainbow triangle in every proper edge-colouring is simply the threshold for containing a triangle, which is known to be $1/n$.

	\subsection{Our results}\label{sec:our-results}

		For a real $d > 0$, we say that $\sG_{d,n} \cup \mathbb{G}(n,p)$ a.a.s.\ satisfies a graph property $\P$, if 
		$$
		\lim_{n \to \infty} \Pro[G_n \cup \mathbb{G}(n,p) \in \P] = 1
		$$
		holds for {\sl every} sequence $\{G_n\}_{n \in \mathbb{N}}$ satisfying $G_n \in \sG_{d,n}$ for every $n \in \mathbb{N}$. We say that $\sG_{d,n} \cup \mathbb{G}(n,p)$ a.a.s.\ does not satisfy $\P$, if 
		$$
		\lim_{n \to \infty} \Pro[G_n \cup \mathbb{G}(n,p) \in \P] = 0
		$$
		holds for at {\sl least} one sequence $\{G_n\}_{n \in \mathbb{N}}$ satisfying $G_n \in \sG_{d,n}$ for every $n \in \mathbb{N}$. Throughout, we suppress this sequence-based terminology and write more concisely that $\sG_{d,n} \cup \mathbb{G}(n,p)$ a.a.s.\ satisfies (or does not) a certain property. In particular, given a fixed graph $H$, we write that a.a.s.\ $\sG_{d,n} \cup \mathbb{G}(n,p) \rainbow H$ to mean that for every sequence $\{G_n\}_{n \in \mathbb{N}}$, satisfying $G_n \in \sG_{d,n}$ for every $n \in \mathbb{N}$, the property $G_n \cup \mathbb{G}(n,p) \rainbow H$ holds asymptotically almost surely. On the other hand, we write that a.a.s.\ $\sG_{d,n} \cup \mathbb{G}(n,p) \notrainbow H$ to mean that there exists a sequence $\{G_n\}_{n \in \mathbb{N}}$, satisfying $G_n \in \sG_{d,n}$ for every $n \in \mathbb{N}$, for which a.a.s.\ $G_n \cup \mathbb{G}(n,p) \rainbow H$ does not hold.

		A sequence $\widehat{p}:=\widehat{p}(n)$ is said to form a {\em threshold} for the property $\P$ in the perturbed model, if $\sG_{d,n} \cup \mathbb{G}(n,p)$ a.a.s.\ satisfies $\P$ whenever $p = \omega(\widehat{p})$, and if 
		$\sG_{d,n} \cup \mathbb{G}(n,p)$ a.a.s.\ does not satisfy $\P$  whenever $p = o(\widehat{p})$.

		For every real $d > 0$ and every pair of integers $s,t \geq 1$, every sufficiently large graph $G \in \sG_{d,n}$ satisfies $G \rainbow K_{s,t}$; in fact, every proper colouring of $G$ supersaturates $G$ with $\Omega(n^{s+t})$ rainbow copies of $K_{s,t}$. The latter is a direct consequence of~\eqref{eq:non-rainbow} stated below (see also~\cite{KMS07}). Consequently, the property $\sG_{d,n} \cup \mathbb{G}(n,p) \rainbow K_{s,t}$ is trivial as no random perturbation is needed for it to be satisfied. The emergence of rainbow copies of non-bipartite prescribed graphs may then be of interest. For odd cycles (including $K_3$) we prove the following.

		\begin{proposition} \label{thm:main:odd-cycles}
			For every integer $\ell \geq 2$, and every real $0 < d \leq 1/2$, the threshold for the property $\sG_{d,n} \cup \mathbb{G}(n,p) \rainbow C_{2 \ell - 1}$ is $n^{-2}$.
		\end{proposition}

		Unlike the aforementioned thresholds for the property $\mathbb{G}(n,p) \rainbow C_\ell$, established in~\cite{BCMP19,NPSS17}, the threshold for the counterpart property in the perturbed model is independent of the length of the cycle. The proof of Proposition~\ref{thm:main:odd-cycles} is fairly standard and is thus postponed until Section~\ref{sec:odd-cylces}. 

		\smallskip

		Our main result concerns the thresholds for the emergence of rainbow complete graphs in properly coloured randomly perturbed dense graphs. The aforementioned results of~\cite{KMPS18,NPSS17} can be easily used in order to establish a lower bound on such thresholds. Indeed, by these results, if $r \geq 5$ and $p = o \left(n^{-1/m_2(K_r)} \right)$, then a.a.s.\ there exists a proper colouring of the edges of $\mathbb{G}(n,p)$ admitting no rainbow copy of $K_r$. Consequently, given a real number $0 < d \leq 1/2$ and an $n$-vertex bipartite graph $G$ of edge-density $d$, a.a.s.\ there exists a proper edge-colouring of $G \cup \mathbb{G}(n,p)$ admitting no rainbow copy of $K_{2r-1}$, provided that $p = o \left(n^{-1/m_2(K_r)} \right)$. We conclude that $\sG_{d,n} \cup \mathbb{G}(n,p) \notrainbow K_{2r}$ and $\sG_{d,n} \cup \mathbb{G}(n,p) \notrainbow K_{2r-1}$ hold a.a.s.\ whenever $p = o \left(n^{-1/m_2(K_r)} \right)$. 

		For every $r \geq 5$, we prove a matching upper bound for the above construction. Our main result reads as follows. 

		\begin{theorem}\label{thm:main}
			Let a real number $0 < d \leq 1/2$ and an integer $r \geq 5$ be given. 
			Then, the threshold for the property $\sG_{d,n} \cup \mathbb{G}(n,p) \rainbow K_{2r}$ is $n^{-1/m_2(K_r)}$. In fact,  $\sG_{d,n} \cup \mathbb{G}(n,p)$ a.a.s.\ has the property that every proper colouring of its edges gives rise to $\Omega\left(p^{2\binom{r}{2}}n^{2r}\right)$ rainbow copies of $K_{2r}$, whenever $p = \omega(n^{-1/m_2(K_r)})$. 
		\end{theorem}
		
		The following result is an immediate consequence of Theorem~\ref{thm:main} and of the aforementioned lower bound.

		\begin{corollary}\label{cor:main:odd}
			Let a real number $0 < d \leq 1/2$ and an integer $r \geq 5$ be given. 
			Then, the threshold for the property $\sG_{d,n} \cup \mathbb{G}(n,p) \rainbow K_{2r-1}$ is $n^{-1/m_2(K_r)}$. 
		\end{corollary}

		Theorem~\ref{thm:main} and Corollary~\ref{cor:main:odd} establish that for {\sl sufficiently large} complete graphs, i.e., $K_s$ with $s \geq 9$, the threshold for the property $\sG_{d,n} \cup \mathbb{G}(n,p) \rainbow K_s$ is governed by a single parameter, namely, $m_2(K_{\ceil{s/2}})$. This turns out to be true (almost, at least) for $s = 8$ as well, but proving it requires new ideas. For $4 \leq s \leq 7$, this is not the case; here, for each value of $s$ in this range, the threshold is different. Using completely different methods, we prove the following results in the companion paper~\cite{ADHLsmall}. 
		
		\begin{theorem}\label{thm:main:457}
			Let $0 < d \leq 1/2$ be given. 
			\begin{enumerate}
				\item The threshold for the property $\sG_{d,n} \cup \mathbb{G}(n,p) \rainbow 
							K_4$ is $n^{-5/4}$.
							
				\item The threshold for the property $\sG_{d,n} \cup \mathbb{G}(n,p) \rainbow 
							K_5$ is $n^{-1}$. 
							
				\item The threshold for the property $\sG_{d,n} \cup \mathbb{G}(n,p) \rainbow 
							K_7$ is $n^{-7/15}$. 
					
			\end{enumerate}
		\end{theorem}

		For $K_6$ and $K_8$ we can ``almost'' determine the thresholds.

		\begin{theorem}\label{thm:main:6}
			Let $0 < d \leq 1/2$ be given. 
			\begin{enumerate}
				\item 
					The property $\sG_{d,n} \cup \mathbb{G}(n,p) \rainbow K_6$ holds a.a.s.\ whenever $p = \omega(n^{-2/3})$. 
				\item 
					For every constant $\eps > 0$ it holds that a.a.s.\ $\sG_{d,n} \cup \mathbb{G}(n,p) \notrainbow K_6$ whenever $p := p(n) = n^{-(2/3 + \eps)}$. 
			\end{enumerate}
		\end{theorem}

		\begin{theorem}\label{thm:main:8}
			Let $0 < d \leq 1/2$ be given. 
			\begin{enumerate}
				\item  
					The property $\sG_{d,n} \cup \mathbb{G}(n,p) \rainbow K_8$ holds a.a.s.\ whenever $p = \omega(n^{-2/5})$. 
				\item 
					For every constant $\eps > 0$ it holds that a.a.s.\ $\sG_{d,n} \cup \mathbb{G}(n,p) \notrainbow K_8$ whenever $p := p(n) = n^{-(2/5 + \eps)}$. 
			\end{enumerate}
		\end{theorem}

		Note that Part 1 of Theorem~\ref{thm:main:8} follows directly from our results in this paper (see Proposition~\ref{prop:even-1-statement} in Section~\ref{sec::mainproof}). Part 2 however, requires new ideas which are detailed in~\cite{ADHLsmall}.

\subsection{Overview of the proof of Theorem~\ref{thm:main}}

In this section, we provide a brief overview of the main ideas involved in our proof of the main result of this paper, namely Theorem~\ref{thm:main}. We consider the emergence of rainbow copies of $K_{2r}$ in $G \cup R$, where  $G \in \sG_{d,n}$ with $n$ sufficiently large, $0 < d \leq 1/2$ is fixed, and $R \sim \mathbb{G}(n,p)$ with $p:= p(n)$ as in Theorem~\ref{thm:main}. Without loss of generality (and owing to Szemer\'edi's regularity lemma~\cite{Szemeredi78}), the graph $G$ can be assumed to be an $\eps$-regular bipartite graph (see Section~\ref{sec:pre} for a definition) with vertex bipartition $V(G) = U \discup W$ such that $|U| = |W| = m = \Omega(n)$. Our argument entails the exposure of $R \sim \mathbb{G}(n,p)$, first over $W$ and then over $U$ as described below. 

Prior to exposing $R$ over $W$, we prove that $R[W]$ a.a.s.\ satisfies simultaneously the three properties detailed below in Claim~\ref{clm:round1-props}; at this stage not all properties stated in that claim can be motivated. The third property asserts that no matter how $R[W]$ is properly edge-coloured, $\Omega \left(p^{\binom{r}{2}}n^r \right)$ rainbow copies of $K_r$ arise in $R[W]$; moreover, all such rainbow copies are realised over a {\sl predetermined} collection of so-called {\sl desirable} $r$-subsets of $W$. By desirable we mean $r$-subsets whose common neighbourhood in $U$ with respect to $G$ is linear in $n$. Such a {\sl rainbow supersaturation} result, targeting a predetermined collection of $r$-subsets, is proved with the aid of Proposition~\ref{prop:Yoshi} (implicit in~\cite{KKM14}) and  the so-called K{\L}R-Theorem, namely Theorem~\ref{thm:KLR}~\cite[Theorem~1.6(i)]{CGSS14} stated below.

We expose $R$ over $W$ and may assume from now on that the part exposed has the aforementioned properties. 
We then prove {\sl deterministically} that given {\sl any} proper edge-colouring of $G$ and given any rainbow copy $K$ of $K_r$ supported on a desirable $r$-subset of $W$, there are $(1 - o(1))\binom{|N_G(V(K))|}{r}$ many $r$-sets $Y$ in the common neighbourhood $N_G(V(K))$ of $V(K)$ in $G$ having the property that $K \cup G[V(K), Y]$ is rainbow under the given colouring; this is seen in~\eqref{eq:many-rainbow-bip}.

Put another way, as soon as we expose $R$ over $W$, any proper edge-colouring of $R[W] \cup G$ gives rise to $\Omega \left(p^{\binom{r}{2}} n^r \right)$ rainbow copies of $K_r$, all supported on desirable $r$-subsets of $W$ (with respect to $G$). Moreover, any such rainbow copy $K$ has the property that $K \cup G[V(K),Y]$ is rainbow for virtually all $r$-subsets $Y$ of its common neighbourhood.

Prior to exposing $R$ over $U$, we prove that $R[U]$ a.a.s.\ simultaneously satisfies the four properties detailed below in Claim~\ref{clm:round2-props}. Being highly  technical in nature, we are unable to state these here in any meaningful way. Collectively, these properties enable us to appeal yet again to rainbow supersaturation properties of $R[W]$ so as to obtain, given a rainbow copy $K$ of $K_r$ in $W$, a sufficient number of rainbow copies $K'$ of $K_r$ in the common neighbourhood of $V(K)$ in $U$ for which $K \cup G[V(K), V(K')]$ is also rainbow. We show that this holds a.a.s.\ for every proper edge-colouring of $G \cup R$.

At this stage, we have an adequate number of pairs $(K,K') \in W \times U$ such that $K'$ is rainbow and $K \cup G[V(K),V(K')]$ is rainbow. However, there could still be colour clashes between the colours seen on $E(K')$ and $E(K)$ or between those seen on $E(K')$ and those found on $E(G[V(K),V(K')]$; in either of these cases the pair $(K, K')$ is said to be \emph{bad}. The last part of the proof is a {\sl deterministic} argument asserting that there exist many pairs $(K,K')$ as above that are not bad and thus induce a rainbow copy of $K_{2r}$.

\section{Preliminaries}\label{sec:pre}

	Throughout, we make appeals to both the {\sl dense} regularity lemma~\cite{Szemeredi78} (see also~\cite{KS96}) and the {\sl sparse} regularity lemma~\cite{K97} (see also~\cite{GS05}). For a bipartite graph $G := (U \discup W, E)$ and two sets $U' \subseteq U$ and $W' \subseteq W$, write $d_G(U',W') := \frac{e_G(U',W')}{|U'||W'|}$ for the edge-density of the induced subgraph $G[U',W']$. For $0 < p \leq 1$, the graph $G$ is called $(\eps,p)$-{\em regular} if 
	$$
	|d_G(U',W') - d_G(U,W)| < \eps p 
	$$
	holds whenever $U' \subseteq U$ and $W' \subseteq W$ satisfy $|U'| \geq \eps |U|$ and $|W'| \geq \eps |W|$. We abbreviate $(\eps,1)$-regular to $\eps$-{\em regular}. 

	\smallskip

	We make repeated use of a result by Janson~\cite{Janson98} (see also~\cite[Theorems~2.14 and~2.18]{JLR}), regarding random variables of the form $X = \sum_{A \in \sS}I_A$. Here, $\sS$ is a family of non-empty subsets of some ground set $\Omega$ and $I_A$ is the indicator random variable for the event $A \subseteq \Omega_p$, where $\Omega_p$ is the so-called {\sl binomial random set} arising from including every element of $\Omega$ independently at random with probability $p$. For such random variables, set $\lambda := \Ex[X]$,  and define
	$$
	\overline{\Delta} := \sum_{A,B \in \sS} \sum_{A \cap B \neq \emptyset} \Ex[I_A I_B], \quad 
	\Delta :=  \frac{1}{2} \sum_{\stackrel{A,B \in \sS}{A \neq B}} \sum_{A \cap B \neq \emptyset} \Ex[I_A I_B].
	$$
	Janson's result concerns the lower tail of $X$. 

	\begin{theorem}\label{thm:Janson-tail} {\em~\cite[Theorem~2.14]{JLR}}
		For every $0 < t \leq \Ex[X]$, it holds that
		$$
		\Pro[X \leq \Ex[X] - t] \leq \exp(- t^2/2\overline{\Delta}).
		$$
	\end{theorem}

	%

	\medskip

	We also use the following lemma (see, e.g.,\ Lemma 2.1 in~\cite{FS11}), which is known as a basic \emph{dependent random choice} lemma.
		\begin{lemma} \label{lem:dependent-random-choice}
			Let $a, \overline{d}, m, n, r$ be positive integers. Let $G$ be a graph on $n$ vertices with average degree at least $\overline{d}$. If there exists a positive integer $t$ such that
			\[
				\frac{(\overline{d})^t}{n^{t-1}} - \binom{n}{r} \left(\frac{m}{n}\right)^t \geq a,
			\]
			then $G$ contains a set of vertices $U$ of size $|U| \geq a$ such that every $r$ vertices in $U$ have at least $m$ common neighbours.
		\end{lemma}

	We conclude this section with some additional notation which will be used throughout the paper. Given a sequence $f := f(n)$ and constants $\eps_1, \ldots, \eps_k > 0$ independent of $n$, we write $\Omega_{\eps_1,\ldots,\eps_k}(f)$, $\Theta_{\eps_1,\ldots,\eps_k}(f)$, and $O_{\eps_1,\ldots,\eps_k}(f)$ to mean that the constants which are implicit in the asymptotic notation depend on $\eps_1, \ldots, \eps_k$. If $g := g(n)$ is a sequence, then we sometimes write $f \gg g$ and $f \ll g$ to mean $f =  \omega(g)$ and $f = o(g)$, respectively. In addition, given two constants $\mu > 0$ and $\nu > 0$ we write $\mu \ll \nu$ to mean that, while $\mu$ and $\nu$ are fixed,  they can be chosen so that $\mu$ is arbitrarily smaller than $\nu$. 

\section{Properties of $\Gnp$}

	In this section, we collect the various properties of $\Gnp$ facilitating subsequent  arguments. Beyond the aggregation of such properties, the main result of this section is Proposition~\ref{prop:supersaturation} concerning the supersaturation of rainbow copies of a given fixed graph in $\Gnp$ with respect to any proper edge-colouring. We commence, however, with the more standard properties, some of which will also facilitate the proof of Proposition~\ref{prop:supersaturation}. 

	\subsection{Concentration results}

		While Theorem~\ref{thm:main} deals only with complete graphs, in this section we consider a more general class of graphs. 
Recall that a graph $H$ is \emph{strictly $2$-balanced} if $m_2(H) > m_2(K)$ whenever $K \subsetneq H$.
		Throughout this section, $H$ denotes a fixed strictly $2$-balanced graph. 
		For such a graph $H$, let $\HH := \HH_n$ denote the family of (labelled) copies of $H$ in $K_n$. For every $\tilde H \in \HH$, let $Z_{\tilde H}$ denote the indicator random variable for the event $\tilde H \subseteq \Gnp$. Then, $X_H := \sum_{\tilde H \in \HH} Z_{\tilde H}$ counts the number of copies of $H$ in $\Gnp$. Note that 
		\begin{equation} \label{eq:expectation-H}	
			\Ex (X_H) = \sum_{\tilde H \in \HH} p^{e(\tilde H)} = p^{e(H)} \binom{n}{v(H)} \frac{(v(H))!}{|Aut(H)|} = \Theta \left(p^{e(H)} n^{v(H)} \right),
		\end{equation}
		where $Aut(H)$ is the automorphism group of $H$.

		We require large deviation inequalities for both the upper and lower tails of $X_H$. For the lower tail, we make the standard appeal to Janson's inequality (seen at~\eqref{eq:Janson} below) so as to subsequently yield Lemma~\ref{lem:prescribed-copies}. For the upper tail, however, the standard appeal to Chebyshev's inequality is insufficient for our needs. For indeed, subsequent arguments require that certain properties of $\Gnp$  hold with probability at least $1 - \Omega \left(n^{-b} \right)$, for some constant $b$ which we are allowed to choose to be sufficiently large. We thus replace the standard appeal to Chebyshev's inequality with an appeal to one of the main results of Vu's paper~\cite{Vu01}. 

		In broad terms, \cite[Theorem~2.1]{Vu01} asserts that if 
		$$
			p \gg (\log n)^{\frac{v(H)-1}{e(H)}} \cdot n^{- \frac{v(H)}{e(H)}},
		$$
		(i.e., $p$ is larger than the containment threshold for $H$ in $\Gnp$ by at least some polylogarithmic multiplicative factor), then a large deviation inequality for the upper tail of $X_H$ with a decaying exponential error rate exists. As our focus is on strictly $2$-balanced graphs and on $p = \Omega \left(n^{-1/m_2(H)} \right)$, we make do with the following more relaxed formulation of the aforementioned result of Vu~\cite{Vu01}. 

		\begin{theorem} \label{thm:Vu}{\em ~\cite[Theorem~2.1]{Vu01}}
			For any (fixed) $\alpha > 0$ and any strictly $2$-balanced graph $H$, there exists a constant $C_{\ref{thm:Vu}} > 0$, such that  
			$$
			\Pro[X_H \geq (1 + \alpha) \Ex (X_H)] \leq \exp \left(- \Omega_{\alpha,H} \left((\Ex (X_H))^{1/(v(H)-1)}\right)\right)
			$$
			holds whenever $p \geq C_{\ref{thm:Vu}} n^{- 1/m_2(H)}$. 
		\end{theorem}

		We proceed to deal with the lower tail of $X_H$ as outlined above. Writing $H' \sim H''$ whenever $(H', H'') \in \HH \times \HH$ are \emph{not} edge disjoint, let 
		$$
			\bar{\Delta}(H) := \sum_{\substack{(H', H'') \in \HH \times \HH \\ H' \sim H''}} \Ex[Z_{H'} Z_{H''}] = \sum_{\substack{(H', H'') \in \HH \times \HH \\ H' \sim H''}} p^{e(H') + e(H'') - e(H' \cap H'')}. 
		$$
		For a strictly $2$-balanced graph $H$, it is well-known that
		\begin{equation} \label{eq:Delta=expectation}
			\bar{\Delta}(H) = O_H \left(p^{e(H)} n^{v(H)}\right),
		\end{equation}
		whenever $p = \Omega \left(n^{- 1/m_2(H)} \right)$. 

		Janson's inequality, as seen in Theorem~\ref{thm:Janson-tail}, asserts that 
		$$
			\Pro[X_H \leq (1 - \xi) \Ex (X_H)] \leq e^{- \frac{\xi^2 (\Ex (X_H))^2}{2 \bar{\Delta}(H)}}
		$$
		holds for every fixed $0 \leq \xi \leq 1$.
		It then follows by~\eqref{eq:Delta=expectation} and~\eqref{eq:expectation-H} that 
		\begin{equation}\label{eq:Janson}
		\Pro\left[X_H \leq (1 - \xi) p^{e(H)} \binom{n}{v(H)} \frac{(v(H))!}{|Aut(H)|} \right] \leq e^{- \Omega \left(p^{e(H)} n^{v(H)} \right)}
		\end{equation}
		for every fixed $0 \leq \xi \leq 1$.

		We require a slight strengthening of~\eqref{eq:Janson}. Given a set $\C \subseteq \binom{[n]}{v(H)}$ satisfying $|\C| = \eta\binom{n}{v(H)}$ for some fixed $\eta >0$, write $X_H(\C)$ to denote the number of copies of $H$ in $\Gnp$ supported on the members of $\C$, that is, $X_H(\C) = \{\tilde H \in \HH : V(\tilde H) \in \C \textrm{ and } \tilde H \subseteq \Gnp\}$. Then, $\Ex[X_H(\C)] = \Theta_{H,\eta}(p^{e(H)}n^{v(H)})$. Set
		\begin{equation}\label{eq:Delta-C}
			\bar{\Delta}(H,\C) := \sum_{\substack{(H', H'') \in \HH(\C) \times \HH(\C) \\ H' \sim H''}} \Ex[Z_{H'} Z_{H''}],
		\end{equation}
		where here $\HH(C)$ serves as the analogue of $\HH$ for the copies of $H$ supported on $\C$. Since, clearly, $\bar{\Delta}(H,\C) \leq \bar{\Delta}(H)$, inequality~\eqref{eq:Janson} can be extended so as to yield 
		\begin{equation}\label{eq:targeted-Janson}
		\Pro\left[X_H(\C) \leq (1 - \xi) p^{e(H)} |\C| \frac{(v(H))!}{|Aut(H)|} \right] \leq e^{- \Omega \left(p^{e(H)} n^{v(H)} \right)}
		\end{equation}
		holds for every fixed $0 \leq \xi \leq 1$. The following is then established. 

		\begin{lemma} \label{lem:prescribed-copies}
			Let $H$ be a strictly $2$-balanced graph and let $\xi > 0$. There exists an integer $n_0$ such that for every $n \geq n_0$ the following holds. Let $\eta > 0$ be a constant and let $\C \subseteq \binom{[n]}{v(H)}$ satisfying $|\C| \geq \eta \binom{n}{v(H)}$ be fixed. Then, there exists a constant $C_{\ref{lem:prescribed-copies}} > 0$ such that with probability at least $1 - e^{- \Omega \left(p^{e(H)} n^{v(H)} \right)}$, the random graph $\Gnp$ has at least $(1 - \xi) p^{e(H)} |\C| \frac{(v(H))!}{|Aut(H)|}$ copies of $H$ supported on the members of $\C$, whenever $p \geq C_{\ref{lem:prescribed-copies}} n^{- 1/m_2(H)}$. 
		\end{lemma}

		The exponential rate of decay, seen in the error probability of Lemma~\ref{lem:prescribed-copies}, will be used in subsequent applications where we will need a union bound to be extended over a large family, as specified in the following corollary. 

		\begin{corollary} \label{cor:prescribed-copies}
			Let $H$ be a strictly $2$-balanced graph and let $\xi > 0$. There exists an integer $n_0$ such that for every $n \geq n_0$ the following holds. Let $\eta > 0$ be a constant and let a non-empty set $\sC$ comprising of at most $2^{O(n \log n)}$ sets $\C \subseteq \binom{[n]}{v(H)}$, each satisfying $|\C| \geq \eta \binom{n}{v(H)}$, be fixed. Then, there exists a constant $C_{\ref{cor:prescribed-copies}} >0$ such that with probability at least $1 - e^{O \left(n \log n \right) - \Omega \left(p^{e(H)} n^{v(H)} \right)}$, the random graph $\Gnp$ has at least $(1 - \xi) p^{e(H)} |\C| \frac{(v(H))!}{|Aut(H)|}$ copies of $H$ supported on the members of $\C$, for every $\C \in \sC$, whenever $p \geq C_{\ref{cor:prescribed-copies}} n^{- 1/m_2(H)}$. 
		\end{corollary}
		 
		Corollary~\ref{cor:prescribed-copies} is meaningful as long as $H$ and $p$ are such that $p^{e(H)}n^{v(H)} \gg n \log n$ holds; this inequality clearly holds if $H$ is strictly $2$-balanced and $p = \Omega \left(n^{- 1/m_2(H)} \right)$. Nevertheless, in Corollary~\ref{cor:prescribed-copies}, we keep the error probability in its explicit form in order to facilitate subsequent arguments.

	\subsection{Rainbow supersaturation in $\Gnp$}\label{sec:supersaturation}

		The main result of this section, is a {\sl supersaturation} version of the main result of~\cite{KKM14}, and can be seen in Proposition~\ref{prop:supersaturation} below. In its simplest form, Proposition~\ref{prop:supersaturation} asserts that given a strictly $2$-balanced graph $H$, the random graph $\Gnp$ a.a.s.\ has the property that every proper colouring of its edges admits $\Omega(\Ex (X_H))$ rainbow copies of $H$. The formulation of Proposition~\ref{prop:supersaturation} is somewhat more involved as in subsequent arguments we require supersaturation of rainbow copies supported on the members of prescribed subsets of $\binom{[n]}{v(H)}$ and, moreover, we require $\Gnp$ to satisfy the aforementioned property with ``very high'' probability. Our proof of Proposition~\ref{prop:supersaturation} employs the so-called K{\L}R-theorem~\cite[Theorem~1.6(i)]{CGSS14} and the core so-called `technical' result of~\cite[Section~5]{KKM14} (see Proposition~\ref{prop:Yoshi} below). 

		Prior to proving Proposition~\ref{prop:supersaturation}, an explanation as to our appeal to the K{\L}R-theorem is warranted. To this end, let us consider the task of establishing supersaturation of rainbow copies of a prescribed graph $H$ in a host graph $G$, without the additional restriction imposed by Proposition~\ref{prop:supersaturation}, mandating that these copies all be supported on a pre-chosen set of $v(H)$-sets. 
		A moment's thought\footnote{This argument is reproduced below in the proof of Observation~\ref{obs:non-rainbow-Krr}.}, reveals that upon fixing a proper edge-colouring of $G$, the number of non-rainbow copies of $H$ can be upper bounded by

		\begin{equation} \label{eq:non-rainbow}
		e(G) \cdot n \cdot \max_{\substack{e,f \in E(G)\\ e\, \cap\, f = \emptyset}} \left|\{e,f\} \extends{G}{H}\right|,
		\end{equation}
		where $\{e,f\} \extends{G}{H}$ denotes the set of injections of the form $H \mapsto G$ constrained to containing the listed pair of edges of $G$, namely $e$ and $f$. Put another way, this set is comprised of all the so-called {\em extensions} of the pair of fixed edges $\{e,f\}$ into a (labelled) copy of $H$ in $G$. This type of argument can also be seen in~\cite{KMS07}, where it is attributed to~\cite{AJMP03}. 
		   
		Employing~\eqref{eq:non-rainbow} in order to establish Proposition~\ref{prop:supersaturation} entails attaining sufficiently tight upper bounds on $\left| \{e,f\} \extends{\Gnp}{H} \right|$ for any pair of independent edges in $\Gnp$ for the relevant values of $p$. In our case, $p$ can be as low as to yield $\Ex \left(\left|\{e,f\} \extends{\Gnp}{H} \right|\right) = o(1)$; rendering standard concentration-type arguments for estimating the order of magnitude of $\left|\{e,f\} \extends{\Gnp}{H}\right|$ meaningless. We circumvent this obstacle by resorting to a more detailed analysis of counting non-rainbow copies of $H$ in $\Gnp$; the latter approach, as mentioned above, entails the use of the K{\L}R-theorem. Nevertheless,~\eqref{eq:non-rainbow} remains relevant as will be seen in the sequel. 



		Let $V(H) = [h] := \{1, \ldots, h\}$. Following~\cite{CGSS14}, we write $\G(H,\ell,m,p,\eps)$ to denote the collection of graphs $\Gamma$ obtained as follows. 
		The vertex set of $\Gamma$ is 
		$
		V(\Gamma) = V_1 \discup \cdots \discup V_h
		$,
		where $|V_i| = \ell$ for every $i \in [h]$. For every edge $ij \in E(H)$, add an $(\eps,p)$-regular graph with at least $m$ edges between the pair $(V_i, V_j)$; these are the sole edges of $\Gamma$. For such a graph $\Gamma$, a copy of $H$ in $\Gamma$ is called {\em canonical} if it has a single vertex in each $V_i$. We write $\Gamma(H)$ to denote the collection of canonical copies of $H$ in $\Gamma$. 

		\begin{remark}
			When $H$ is a complete graph (as will be the case later on), every copy of $H$ in $\Gamma$ is a canonical copy of $H$.
		\end{remark}

		The following result is implicit in~\cite{KKM14}.

		\begin{proposition} \label{prop:Yoshi} 
			For every graph $H$ and every real number $b > 0$, there exist a constant $\beta_{\ref{prop:Yoshi}} > 0$ and an integer $n_0 > 0$ such that the following holds for every (fixed) $\eps > 0$. If $n \geq n_0$ and $p := p(n) \geq C_{\ref{prop:Yoshi}} \log n/n$, where $C_{\ref{prop:Yoshi}} > 0$ is an appropriately chosen constant, then $\Gnp$ satisfies the following property with probability at least $1 - \Omega \left(n^{-b} \right)$. Every proper colouring $\psi$ of the edges of $\Gnp$ gives rise to a subgraph $\Gamma_\psi \subseteq \Gnp$ satisfying $\Gamma_\psi \in \G(H,\floor{n/v(H)}, \floor{\beta_{\ref{prop:Yoshi}} p \floor{n/v(H)}^2}, p, \eps)$; moreover, every member of $\Gamma_\psi(H)$ is rainbow under $\psi$. 
			
		\end{proposition}

		\begin{remark}
			{\em The assertion of Proposition~\ref{prop:Yoshi} can be found in the proof of Lemma 5.3 in~\cite{KKM14}. This lemma relies on various additional results. In order to obtain the bound $1 - \Omega \left(n^{-b} \right)$ stated in Proposition~\ref{prop:Yoshi}, one has to verify that the assertion of Lemma 3.3 from~\cite{KKM14} holds with this probability (as opposed to simply a.a.s.\ as is stated there). 
The latter asserts that certain properties are a.a.s. satisfied by $G \sim \mathbb{G}(n,p)$ for an appropriate choice of $p$. Lemma~5.3 of~\cite{KKM14} is a deterministic lemma that delivers the graph $\Gamma_\psi$, defined in Proposition~\ref{prop:Yoshi}, in any graph satisfying the properties stated in~\cite[Lemma~3.3]{KKM14}.  			
The aforementioned bound $1 - \Omega \left(n^{-b} \right)$ on the success probability of~\cite[Lemma~3.3]{KKM14} can indeed be attained and this can be traced back to~\cite[Theorem~1.1]{FO05}, which is used in the proof of~\cite[Lemma~3.3]{KKM14}.}   
		\end{remark}

		The argument of~\cite{KKM14} entails an application of the so-called {\em embedding lemma} associated with the K{\L}R-theorem (see~\cite[Lemma~3.9]{KKM14}) to $\Gamma_\psi$; thus ensuring at least one rainbow copy of $H$. As our aim is set on supersaturation of rainbow copies of $H$, we replace~\cite[Lemma~3.9]{KKM14} with an application of the so-called {\em one-sided counting lemma} associated with the K{\L}R-theorem, namely~\cite[Theorem~1.6(i)]{CGSS14}.

		\begin{theorem} \label{thm:KLR}{\em~\cite[Theorem~1.6(i)]{CGSS14}}
			For every $d > 0$ and every strictly $2$-balanced graph $H$, there exist positive constants $\zeta_{\ref{thm:KLR}}, \xi_{\ref{thm:KLR}}$ and an integer $n_0 > 0$ such that the following holds. For every $\eta > 0$, there exists a constant $C_{\ref{thm:KLR}} > 0$ such that $\Gnp$ admits the following property with probability at least $1 - e^{- \Omega_{H,d,\eta}(pn^2)}$, whenever $p := p(n) \geq C_{\ref{thm:KLR}} n^{- 1/m_2(H)}$ and $n \geq n_0$. For every $\ell \geq \eta n$ and $m \geq d pn^2$, and for every subgraph $\Gamma \subseteq \Gnp$ satisfying $\Gamma \in \G(H,\ell,m,p,\zeta_{\ref{thm:KLR}})$, it holds that
			$$
				|\Gamma(H)| \geq \xi_{\ref{thm:KLR}} \left(\frac{m}{n^2}\right)^{e(H)} \binom{n}{v(H)}.
			$$
		\end{theorem}

		Proposition~\ref{prop:Yoshi} and Theorem~\ref{thm:KLR} imply the following rainbow supersaturation result for $\Gnp$. 

		\begin{corollary} \label{cor:supersaturation}
			For every real $b > 0$ and every strictly $2$-balanced graph $H$, there exist positive constants $\beta_{\ref{cor:supersaturation}}, C_{\ref{cor:supersaturation}}$ and an integer $n_0 > 0$ such that if $n \geq n_0$ and $p := p(n) \geq C_{\ref{cor:supersaturation}} n^{- 1/m_2(H)}$, then with probability at least $1 - \Omega \left(n^{-b} \right)$, every proper edge colouring of the random graph $\Gnp$ admits at least $\beta_{\ref{cor:supersaturation}} p^{e(H)} \binom{n}{v(H)} \frac{(v(H))!}{|Aut(H)|}$ rainbow copies of $H$.
		\end{corollary}

		\begin{proof}
			Let $\beta = \beta_{\ref{prop:Yoshi}}$ be the constant given by Proposition~\ref{prop:Yoshi}, corresponding to $b$.
			Set $d = \beta / \left(2v(H)^2 \right)$, and let $\zeta = \zeta_{\ref{thm:KLR}}$ and $\xi = \xi_{\ref{thm:KLR}}$ be the corresponding constants given by Theorem~\ref{thm:KLR}.
			Let $\eta = v(H)/2$, and suppose that $C_{\ref{cor:supersaturation}}$ is chosen to be at least as large as $C_{\ref{thm:KLR}}$, where the latter constant is given by Theorem~\ref{thm:KLR} in correspondence to $\eta$.
			Let $\ell = \floor{n/v(H)}$ and $m = \floor{\beta p\floor{n/v(H)}^2}$.
			 
			By Proposition~\ref{prop:Yoshi}, with probabilisty at least $1 - \Omega(n^{-b})$, given a proper edge-colouring $\psi$ of $\Gnp$, there exists a subgraph $\Gamma_{\psi} \subseteq \Gnp$ such that $\Gamma_{\psi} \in \G(H, \ell, m, p, \zeta)$ and, moreover, every member of $\Gamma_{\psi}$ is rainbow under $\psi$ (here we used the fact that $p \ge \omega(\log n / n)$ which holds e.g.\ since $H$ is strictly $2$-balanced and thus contains a cycle, implying $m_2(H) > 1$). 

			Observe that $\ell \ge \eta n$ and $m \ge dpn^2$. It thus follows by Theorem~\ref{thm:KLR} that, with probability at least $1 - \Omega(n^{-b})$, the subgraph $\Gamma_{\psi}$ satisfies 
			\begin{equation*}
				|\Gamma_{\psi}(H)| 
				\ge \xi \left(\frac{m}{n^2}\right)^{e(H)} \binom{n}{v(H)}
				\ge \xi d^{e(H)} p^{e(H)} \binom{n}{v(H)}.
			\end{equation*}
			Setting $\beta_{\ref{cor:supersaturation}} = \xi d^{e(H)} \frac{|Aut(H)|}{(v(H))!}$ concludes the proof of Corollary~\ref{cor:supersaturation}.
		\end{proof}

		We are now ready to state and prove the main result of this section. 

		\begin{proposition} \label{prop:supersaturation}
			For every real $b > 0$, every strictly $2$-balanced graph $H$, and every $0 < \beta \ll \beta_{\ref{cor:supersaturation}}(b, H)$, there exist $n_0 > 0$ and $C_{\ref{prop:supersaturation}} > 0$ such that the following holds whenever $n \geq n_0$ and $p := p(n) \geq C_{\ref{prop:supersaturation}} n^{- 1/m_2(H)}$. Fix a non-empty family $\sC$ comprised of at most $2^{O(n \log n)}$ sets $\C \subseteq \binom{[n]}{v(H)}$, each of size $|\C| \geq (1 - \beta) \binom{n}{v(H)}$. Then, with probability at least $1 - \Omega \left(n^{-b} \right)$, every proper edge colouring of the random graph $\Gnp$
			admits $\Omega_{H,\beta} \left(p^{e(H)} n^{v(H)} \right)$ rainbow copies of $H$ supported on members of $\C$, for every $\C \in \sC$.
		\end{proposition}

		\begin{proof}
			Let $b, H$, and $\beta$ be as in the statement of the proposition. 
			Set auxiliary constants $\eta = 1 - \beta$ and $\alpha, \xi \ll \beta$. Define
			$$
			C_{\ref{prop:supersaturation}} := \max \{C_{\ref{thm:Vu}}(\alpha, H), C_{\ref{cor:prescribed-copies}}(\eta, \xi, H), C_{\ref{cor:supersaturation}}(b, H)\}.
			$$

			Then, for sufficiently large $n$ and for $p := p(n) \geq C_{\ref{prop:supersaturation}} n^{- 1/m_2(H)}$, Theorem~\ref{thm:Vu}, Corollary~\ref{cor:prescribed-copies}, and Corollary~\ref{cor:supersaturation} collectively imply that with probability at least 
			$$
			1 - e^{- \Omega_{\alpha,H} \left( \left(p^{e(H)} n^{v(H)} \right)^{1/(v(H)-1)}\right)} - e^{O(n\log n) - \Omega_{\xi, H, \eta} \left(p^{e(H)}n^{v(H)} \right)} - \Omega(n^{- b}) \geq 1 - \Omega(n^{-b})
			$$
			the following properties are satisfied by $G \sim \Gnp$ simultaneously. 
			\begin{enumerate} [labelindent = \parindent, leftmargin = *, label = \bf (P.\arabic*)]
				\item \label{itm:few-H-copies}
					$G$ admits at most $(1 + \alpha) p^{e(H)} \binom{n}{v(H)} \frac{(v(H))!}{|Aut(H)|}$ copies of $H$.

				\item \label{itm:many-H-copies}
					For every $\C \in \sC$, the graph $G$ admits at least $(1 - \xi)(1 - \beta) p^{e(H)} \binom{n}{v(H)} \frac{(v(H))!}{|Aut(H)|}$ 
					copies of $H$ supported on the members of $\C$.
				
				\item \label{itm:many-rainbow-H-copies}
					Every proper colouring of the edges of $G$ admits at least $\beta_{\ref{cor:supersaturation}}(b,H) p^{e(H)} \binom{n}{v(H)} \frac{(v(H))!}{|Aut(H)|}$ rainbow copies of $H$ in $G$. 
			\end{enumerate}

			Let a proper colouring $\psi$ of the edges of $G \sim \Gnp$ and a member $\C \in \sC$ be fixed. It follows by Properties \ref{itm:few-H-copies} and \ref{itm:many-H-copies} that all but at most 
			$$
			[(1 + \alpha) - (1 - \xi) (1 - \beta)] p^{e(H)} \binom{n}{v(H)} \frac{(v(H))!}{|Aut(H)|} = (\alpha + \xi + \beta - \xi \beta) p^{e(H)} \binom{n}{v(H)} \frac{(v(H))!}{|Aut(H)|}
			$$
			of the copies of $H$ in $G$ are supported on members of $\C$. Then, owing to Property \ref{itm:many-rainbow-H-copies}, $\psi$ admits at least 
			$$
			\left(\beta_{\ref{cor:supersaturation}}(b,H) - \alpha - \xi - \beta + \xi \beta\right)p^{e(H)}\binom{n}{v(H)} \frac{(v(H))!}{|Aut(H)|}
			$$
			rainbow copies of $H$ which are supported on members of $\C$. The claim then follows since $\alpha, \xi, \beta \ll \beta_{\ref{cor:supersaturation}}(b, H)$ hold by assumption.
		\end{proof}

		\begin{remark}
			Proposition~\ref{prop:supersaturation} is somewhat of an {\sl overkill} as far as our needs go regarding rainbow supersaturation. In particular, in the sequel, we apply this result over a single set system of $v(K_r)$-sets and not over $2^{O(n \log n)}$ such set-systems. As the difference in the proof is minuscule, we keep the above formulation. 
		\end{remark}

	\subsection{Number of complete subgraphs containing a prescribed vertex}

		The final property of random graphs that we shall use, is that when $p$ is sufficiently large, the number of copies of $K_r$ that contain any single vertex is significantly smaller than the expected total number of copies of $K_r$ in the random graph. 

		Let a vertex $v \in [n]$ be fixed. Let $A_1, \ldots, A_t$, where $t = \binom{n-1}{r-1}$, be an enumeration of the elements of $\binom{[n] \setminus v}{r-1}$. For every $1 \leq i \leq t$, let $Y_i$ denote the indicator random variable for the event: ``$(\Gnp)[A_i \cup \{v\}]$ is a clique''. Let $X_v = \sum_{i=1}^t Y_i$, that is, $X_v$ counts  the number of copies of $K_r$ in $\Gnp$ containing $v$.  

		\begin{lemma} \label{lem:vertices-copies}
			For any fixed integer $r \geq 3$, the random graph $\Gnp$ a.a.s.\ satisfies the property that $X_v = o(\Ex (X_{K_r}))$ holds for every vertex $v \in [n]$, whenever $p := p(n) = \Omega \left(n^{- 1/m_2(K_r)} \right)$.
		\end{lemma}

		\begin{proof}
			Fix an arbitrary vertex $v \in [n]$. Observe that
			\begin{align*}
				\frac{n^{1/r} \Ex (X_v) + \Ex (X_v)}{\Ex (X_{K_r})} \leq \frac{2 n^{1/r} \binom{n-1}{r-1} p^{\binom{r}{2}}}{\binom{n}{r} p^{\binom{r}{2}}} = O_r(n^{1/r - 1}),
			\end{align*}
			implying that $n^{1/r} \Ex (X_v) + \Ex (X_v) = o( \Ex (X_{K_r}))$. Therefore, a union bound over $[n]$ implies that in order to prove the lemma it suffices to show that
			$$
			\Pro [|X_v - \Ex (X_v)| \geq n^{1/r} \Ex (X_v)] = o(1/n).
			$$
			Applying Chebyshev's inequality (see e.g.~\cite{Alon-Spencer, JLR}), we obtain
			$$
			\Pro [|X_v - \Ex (X_v)| \geq n^{1/r} \Ex (X_v)] \leq \frac{\VAR (X_v)}{n^{2/r} (\Ex (X_v))^2}.
			$$
			For the variance of $X_v$ we may write 
			\begin{equation} \label{eq::VarVertex}
			\VAR (X_v) = \sum_{i=1}^t \VAR (Y_i) + 2 \sum_{1 \leq i < j \leq t} \COV(Y_i, Y_j) \leq \Ex (X_v) + O_r \left(n^{2r-3} p^{2 \binom{r}{2} - 1} \right),
			\end{equation}
			where the second term on the right hand side of the above inequality is an upper bound on $\sum_{1 \leq i < j \leq t} \COV(Y_i, Y_j)$; this bound can be proved by observing that the dominant term in this sum arises from pairs of copies of $K_r$ which share $v$ and one other vertex. Then
			$$
			\Pro [|X_v - \Ex (X_v)| \geq n^{1/r} \Ex (X_v)] \leq \frac{1}{n^{2/r} \Ex (X_v)} + \frac{O_r(1)}{n^{2/r} p n} = o(1/n),
			$$
			where the inequality holds by Chebyshev's inequality and by~\eqref{eq::VarVertex}, and the equality follows by a straightforward calculation which uses the assumed lower bound on $p$.
		\end{proof}

\section{Sparse complete bipartite graphs}

	In this section, we consider certain applications of~\eqref{eq:non-rainbow} to sparse bipartite graphs and relatives thereof that arise in subsequent arguments. We start with the following observation.

	\begin{observation}\label{obs:non-rainbow-Krr}
		For every pair of integers $r \geq 1$ and $s \geq 2$, there exists an integer $n_0$ such that for any $n \geq n_0$, every proper colouring of the edges of $K_{r,n}$ admits at most $O_{r,s} (n^{s-1} )$ non-rainbow copies of $K_{r,s}$ (whose partition class of size $r$ coincides with the partition class of size $r$ of $K_{r,n}$). 
	\end{observation}

	\begin{proof}
		Fix a proper edge colouring $\psi$ of $K_{r,n}$. Note first that, since $\psi$ is proper, our claim is trivial if $r = 1$; we can thus assume that $r \geq 2$. Clearly, we may also assume that $n \geq s$. Any non-rainbow copy of $K_{r,s}$ must admit at least two edges bearing the same colour under $\psi$. The number of ways to pick the first of these two edges is upper bounded by $e(K_{r,n}) = r n$. The number of ways to choose the second of these two edges is at most $r-1$ as the colouring is proper and one of the bipartition classes of the graph has size $r$. The number of ways to complete any such choice of two edges into a copy of $K_{r,s}$ in $K_{r,n}$ is $\binom{n-2}{s-2}$; the latter quantity making sense owing to $n \geq s \geq 2$. We conclude that the number copies of $K_{r,s}$ in $K_{r,n}$ (whose partition class of size $r$ coincides with the partition class of size $r$ of $K_{r,n}$) that are non-rainbow under $\psi$ is at most $r n \cdot (r-1) \cdot \binom{n-2}{s-2} = O_{r,s} (n^{s-1})$.   
	\end{proof}

	For two integers $n >r$, let $\widehat{K}_{r,n}$ denote the graph obtained from $K_{r,n}$ by placing a copy of $K_r$, denoted $K$, on its partition class of size $r$. By the {\em bipartition classes of $\widehat{K}_{r,n}$} we mean the bipartition classes of $K_{r,n}$. Given an integer $s \geq 2$, write $\B = \B(s)$ to denote the family of copies of $K_{r,s}$ in $\widehat{K}_{r,n}$ such that its partition class of size $r$ coincides with $V(K)$. 

	\begin{definition}
	Let $\psi$ be a proper colouring of the edges of $\widehat{K}_{r,n}$ under which $K$ is rainbow. A member $B \in \B$ is said to be {\em compatible} with the colouring of $K$ under $\psi$ (or, for brevity, simply {\em compatible} with $K$) if the following two properties hold
	\begin{enumerate}
		\item $B$ is rainbow under $\psi$;
		\item the sets of colours seen on $E(K)$ and $E(B)$ are disjoint; in this case we say 
		that $K$ and $B$ do not {\em clash} under $\psi$. 
	\end{enumerate}
	\end{definition}

	\begin{observation}\label{obs:compatible}
		Let $r, s$, and $n$ be as in Observation~\ref{obs:non-rainbow-Krr}. If $\psi$ is a proper colouring of the edges of $\widehat{K}_{r,n}$ under which $K$ is rainbow, then all but at most $O_{r,s}(n^{s-1})$ of the members of $\B$ are compatible with $K$. 
	\end{observation}

	\begin{proof}
		In view of Observation~\ref{obs:non-rainbow-Krr}, it suffices to prove that there are $O_{r,s}(n^{s-1})$ members of $\B$ that clash with $K$. Let $c$ be a colour appearing on some edge of $K$, and let $X_c$ be the set of vertices in the partition class of size $n$ that send a $c$-coloured edge to $K$; then $|X_c| \le r-2$. Note that every member of $\B$ that clashes with $K$ contains a vertex in $X_c$ for some colour $c$ that appears on $K$. As the number of members in $\B$ that contain a given vertex is at most $\binom{n}{s-1}$, we conclude that the number of members of $\B$ that clash with $K$ is at most $e(K) \cdot (r-2) \cdot  \binom{n}{s-1} = O_{r,s}(n^{s-1})$.
	\end{proof}

\section{Proof of Theorem~\ref{thm:main}: $1$-statement} \label{sec::mainproof}

	In this section, we prove the $1$-statement associated with Theorem~\ref{thm:main}. The proof of the relevant $0$-statement is detailed in Section~\ref{sec:our-results} prior to the statement of the theorem. The main result of this section reads as follows. 

	\begin{proposition}\label{prop:even-1-statement}
		For every real number $d > 0$ and integer $r \geq 3$, the property $\sG_{d,n} \cup \mathbb{G}(n,p) \rainbow K_{2r}$ holds a.a.s., whenever $p := p(n) = \omega(n^{-1/m_2(K_r)})$. In fact, for values of $p$ in this range, $\sG_{d,n} \cup \mathbb{G}(n,p)$ a.a.s.\ has the property that every proper colouring of its edges gives rise to $\Omega\left(p^{2\binom{r}{2}}n^{2r}\right)$ rainbow copies of $K_{2r}$. 
	\end{proposition}

	\begin{proof}
		Fix $d > 0$, an integer $r \geq 3$, and $G \in \sG_{d,n}$, where throughout we assume $n$ to be sufficiently large. Set auxiliary constants 
		\begin{equation} \label{eq::Bandg}
			r \ll b \mand 0<\eps \ll \gamma \ll \beta_{\ref{cor:supersaturation}}(b, K_r).
		\end{equation}

		By a standard application of the (so-called {\sl dense}) regularity lemma~\cite{Szemeredi78} (see also~\cite{KS96}), we may assume, without loss of generality, that $G$ is a bipartite $\eps$-regular graph with edge-density at least $d'$ for some constant $\eps \ll d' \ll \gamma$, such that $V(G) = W \discup U$, and $|U| = |W| = m = \Omega_{\eps}(n)$. Let 
		\begin{equation}\label{eq::CW-def}
			\C_W = \left\{X \in \binom{W}{r}: |N_X| = \Omega_{d',\eps}(m)\right\},
		\end{equation}
		where $N_X := \{u \in U : uv \in E(G) \textrm{ for every } v \in X\}$ is the common neighbourhood of $X$ in $G$. Regularity then implies that 
		\begin{equation} \label{eq::CW}
		|\C_W| \geq (1 - \gamma) \binom{m}{r}.
		\end{equation}

		We expose the random edges added to $G$ in three steps. Firstly, the random edges with both endpoints in $W$ are exposed; secondly, the random edges with both endpoints in $U$ are exposed; thirdly and finally, all other random edges are exposed. Note, however, that the third step is a mere formality as, indeed, the edges exposed in this step serve no role in the formation of any eventual rainbow copy of $K_{2r}$ produced by our argument. 

		\begin{claim}\label{clm:round1-props}
			Asymptotically almost surely $G_1 \sim \Gnp[W]$ satisfies the following properties simultaneously. 
			\begin{enumerate}[labelindent = \parindent, leftmargin = *, label = \bf (Q.\arabic*)]
				\item \label{itm:few-cliques-vertex}
					Every vertex $w \in W$ lies in at most $o \left(p^{\binom{r}{2}}m^r \right)$ copies of $K_r$ in $G_1$.

				\item \label{itm:few-cliques-total}
					$G_1$ has $O \left(p^{\binom{r}{2}} m^r \right)$ copies of $K_r$.
				
				\item \label{itm:many-cliques-CW}
					Every proper colouring $\psi$ of the edges of $G_1$ admits $\Omega \left(p^{\binom{r}{2}} m^r \right)$ rainbow copies of $K_r$ supported on members of $\C_W$. 
			\end{enumerate}
		\end{claim}

		\begin{proof}
			It suffices to prove that $G_1$ satisfies each of the aforementioned properties asymptotically almost surely. Property \ref{itm:few-cliques-vertex} follows from Lemma~\ref{lem:vertices-copies} applied to $G_1$, and Property \ref{itm:few-cliques-total} follows from Theorem~\ref{thm:Vu} applied to $G_1$. 

			Property \ref{itm:many-cliques-CW} follows from Proposition~\ref{prop:supersaturation} applied with $b$ as in~\eqref{eq::Bandg}, $\beta = \gamma$, and $\sC = \{\C_W\}$. Indeed, the latter asserts that $G_1$ a.a.s.\ satisfies the property that any proper colouring of the edges of $G_1$ admits $\Omega \left(p^{\binom{r}{2}} m^r \right)$ rainbow copies of $K_r$ supported on the members of $\C_W$. 
		\end{proof}

		Fix a graph $G_1 \sim (\Gnp)[W]$ satisfying properties \ref{itm:few-cliques-vertex}, \ref{itm:few-cliques-total} and \ref{itm:many-cliques-CW}, and denote $H_1 := G_1 \cup G$. We establish an additional property of $H_1$ deterministically, so to speak; see~\eqref{eq:many-rainbow-bip}. 

		A member $X \in \C_W$ satisfying $H_1[X] \cong K_r$ is termed {\em relevant}. For a relevant $X$, let $\widehat{B}_X$ denote the copy of $\widehat{K}_{r,|N_X|}$ in $H_1$ whose $r$-part is $X$ and whose other part is $N_X$. Given a proper colouring $\psi$ of the edges of $H_1$, a subgraph $K \subseteq H_1$ appearing rainbow under $\psi$ is called $\psi$-{\em rainbow}; similarly, if $H_1[Y]$ is a $\psi$-rainbow clique, then we say that $Y$ is $\psi$-rainbow.

		For a proper edge-colouring $\psi$ of $H_1$ and a relevant $\psi$-rainbow $X \in \C_W$, set 
		$$
		\B_{X,\psi} := \left\{Y \in \binom{N_X}{r}: H_1[X] \cup  H_1[X,Y]\; \text{is $\psi$-rainbow}\right\}. 
		$$
		It follows by Observation~\ref{obs:compatible} that 
		\begin{equation} \label{eq:many-rainbow-bip}
				|\B_{X, \psi}| = \left(1 - O_r\left(\frac{1}{|N_X|}\right) \right)\binom{|N_X|}{r} \geq \left(1- O_{r,d',\eps} \left(\frac{1}{m} \right) \right)\binom{|N_X|}{r}, 
		\end{equation}
		where for the last inequality we appeal to $m$ being sufficiently large and $|N_X| = \Omega_{d',\eps}(m)$.

		\bigskip

		The next claim addresses the properties of the distribution $H_1 \cup (\Gnp)[U]$; the notation of subgraphs being $\psi$-rainbow extends naturally to proper edge-colourings $\psi$ of the latter and subgraphs thereof. 

		\begin{claim} \label{clm:round2-props}
			Let $G_1 \sim (\Gnp)[W]$ satisfying properties \ref{itm:few-cliques-vertex},  \ref{itm:few-cliques-total} and \ref{itm:many-cliques-CW} be fixed, and let $G_2 \sim (\Gnp)[U]$.  
			Then, a.a.s.\ the following properties hold simultaneously for every proper colouring $\psi$ of the edges of $G \cup G_1 \cup G_2$ and every relevant $\psi$-rainbow $X \in \C_W$. Denote $\ell := |N_X|$, $\beta := \beta_{\ref{prop:Yoshi}}(K_r,b)$, $k := \beta \cdot p \ell^2/r^2$, $\mu := \zeta_{\ref{thm:KLR}}(\beta/2, K_r)$, and $\alpha := \left(\frac{\beta}{2 r^2} \right)^{\binom{r}{2}} \xi_{\ref{thm:KLR}}(\beta/(2r^2), K_r)$ (note that some of these parameters depend on $X$). Then
		\begin{enumerate}[labelindent = \parindent, leftmargin = *, label = \bf (Q.\arabic*), start = 4]
			\item \label{itm:few-cliques-U} 
				$G_2$ admits $\Theta \left(p^{\binom{r}{2}} m^r \right)$ copies of $K_{r}$.
			
			\item \label{itm:gamma-exists}
				There is a subgraph $\Gamma_{X,\psi} \subseteq G_2[N_X]$ satisfying  
				\begin{equation}\label{eq:KLR-rainbow-graph}
					\Gamma_{X,\psi} \in 
					\G\left(K_r, \floor{\ell/r}, k, p, \mu/2\right),
				\end{equation}
				such that every copy of $K_r$ in $\Gamma_{X,\psi}$ is $\psi$-rainbow. 

			\item \label{itm:many-cliques-gamma}
				Every subgraph $\Gamma \subseteq G_2[N_X]$ satisfying 
				\begin{equation}\label{eq:KLR-counting-prop}
					\Gamma \in 
					\G\left(K_r,\floor{\ell/r},k/2,p,\mu \right),
				\end{equation}
				admits at least $\alpha p^{\binom{r}{2}}\binom{\ell}{r}$ copies of $K_r$.

			\item \label{itm:many-cliques-BW}
				All but at most $\frac{\alpha}{2} p^{\binom{r}{2}}
				\binom{\ell}{r}$ copies of $K_r$ in $G_2[N_X]$ are supported on 
				members of $\B_{X,\psi}$.
		\end{enumerate}
		\end{claim}

		\begin{proof}
			It suffices to prove that a.a.s.\ each of the above four properties holds for every proper colouring $\psi$ and every relevant $\psi$-rainbow $X$. Property \ref{itm:few-cliques-U} follows from Theorem~\ref{thm:Vu} and Lemma~\ref{lem:prescribed-copies} (with $\C = \binom{U}{r}$ and $\eta = 1$). \smallskip

			Property~\ref{itm:gamma-exists} follows from Proposition~\ref{prop:Yoshi}. Indeed, by the assumptions on $p$, we have $p = \omega(\log n / n)$. Consequently,  Proposition~\ref{prop:Yoshi}, applied to the graph $G_2[N_X]$ with $b$ as in~\eqref{eq::Bandg} and $\eps$ (from Proposition~\ref{prop:Yoshi}) set to $\mu/2$, implies that given $X \in \C_W$ there exists a graph $\Gamma_{X,\psi}$ as specified in~\eqref{eq:KLR-rainbow-graph}, for every proper colouring $\psi$, with probability at least $1 - \Omega_{\eps,r}(n^{-b})$. Consequently, the probability that Property~\ref{itm:gamma-exists} fails for some $\psi$ or $X$ is at most $O(n^r \cdot n^{-b}) = o(1)$, where the equality holds since $r \ll b$ by~\eqref{eq::Bandg}.

			\smallskip

			Property~\ref{itm:many-cliques-gamma} is a consequence of the K{\L}R-theorem, namely, Theorem~\ref{thm:KLR}. By Theorem~\ref{thm:KLR}, applied to $G_2[N_X]$ with $H = K_r$ and with $d$ (per that theorem) set to $\beta/(2 r^2)$ (recall that $\beta = \beta_{\ref{prop:Yoshi}}(K_r, b)$), it follows that with probability at least $1 - e^{- \Omega_{r,\eps,d}(p n^2)}$, every $\Gamma \subseteq G_2[N_X]$ such that $\Gamma \in \G\left(K_r, \floor{\ell/r}, k/2, p, \mu \right)$, satisfies the stipulated counting property associated with the K{\L}R-theorem. The probability that Property~\ref{itm:many-cliques-gamma} fails for some $X \in \C_W$ is at most $O(n^r \cdot e^{-\Omega(pn^2)}) = o(1)$, where the equality holds by the assumed lower bound on $p$.    

			\bigskip

			We now prepare for the proof that Property~\ref{itm:many-cliques-BW} holds asymptotically almost surely. By Theorem~\ref{thm:Vu}, the number of copies of $K_r$ in $G_2[N_X]$ is at most
			\begin{equation} \label{eq:Vu-Kr-psi-X}
				(1 + \alpha/4) \cdot p^{\binom{r}{2}} \binom{\ell}{r},
			\end{equation}
			with probability at least $1 - \exp\left( -\Omega\left(\Big( p^{\binom{r}{2}} n^r\Big)^{\frac{1}{r-1}}\right)\right) \geq 1 - \exp\left(-\Omega\left( n^{\frac{3}{2(r-1)}}\right)\right)$; this inequality holds since $p^{\binom{r}{2}} n^r = \omega(n^2 p) = \omega(n^{3/2})$, where here the first equality is due to the assumed lower bound on $p$ and the second equality holds since $r \geq 3$. Thus, the probability that the number of copies of $K_r$ in $N_X$ is larger than the expression appearing in~\eqref{eq:Vu-Kr-psi-X} for some $X$ is at most $n^r \cdot \exp\left(-\Omega\left( n^{\frac{3}{2(r-1)}}\right)\right) = o(1)$.

			\smallskip

			Gearing up towards an application of Corollary~\ref{cor:prescribed-copies}, let $\widehat{B}_X$ denote the copy of $\widehat{K}_{r, |N_X|}$, with its part of size $r$ being $X$ and its other part being $N_X$. Consider the set 
			$$
				\B_X :=  \{\B_{X, \psi}: \; \text{$\psi$ is a proper edge-colouring of $\widehat{B}_X$ and $X$ is $\psi$-rainbow}\}.
			$$
			As the number of proper colourings of $\widehat{B}_X$ (up to relabeling of the colours) is at most $e(\widehat{B}_X)^{e(\widehat{B}_X)} = 2^{O(n\log n)}$, we have $|\B_X| = 2^{O(n\log n)}$ for every $X \in \C_W$.
			Apply Corollary~\ref{cor:prescribed-copies} to $G_2[N_X]$ with $\sC = \B_X$ along with $\alpha/8$ (as set above) and $\eta = 1- \alpha/8$. Using~\eqref{eq:many-rainbow-bip}, the corollary asserts that, for every $\B_{X, \psi} \in \B_X$, the number of copies of $K_r$ supported on members of $\B_{X, \psi}$ is at least 
			\begin{equation}\label{eq:xi-eta-cnt}
				(1-\alpha/8)^2 \cdot p^{\binom{r}{2}}\binom{\ell}{r} \geq
				(1-\alpha/4) \cdot p^{\binom{r}{2}}\binom{\ell}{r},
			\end{equation}
			with probability at least $1 - \exp\left(O(n\log n) - \Omega\Big(p^{\binom{r}{2}n^r}\Big)\right) = 1 - \exp\left(-\Omega(n^{3/2})\right)$.
			It follows that, with probability at least $1 - n^r \exp\left(-\Omega(n^{3/2})\right) = 1 - o(1)$, the number of copies of $K_r$ supported on members of $\B_{X, \psi}$ is at least the number appearing in~\eqref{eq:xi-eta-cnt} for every $X \in \C_W$ and every proper edge-colouring $\psi$.

			\smallskip 

			With the above two properties of $G_2$ established, Property~\ref{itm:many-cliques-BW} follows deterministically, so to speak. Indeed, by~\eqref{eq:Vu-Kr-psi-X} and~\eqref{eq:xi-eta-cnt}, a.a.s.\ there are at most 
			$$
				\frac{\alpha}{2} \cdot p^{\binom{r}{2}}\binom{\ell}{r}
			$$
			copies of $K_r$ in $G_2[N_X]$ not supported on members of $\B_{X,\psi}$, as required for Property~\ref{itm:many-cliques-BW}. 
		\end{proof}

		Let $H \sim G \cup \Gnp$ satisfying Properties \ref{itm:few-cliques-vertex}--\ref{itm:many-cliques-BW} be fixed. Let $\psi$, a proper colouring of the edges of $H$, be fixed as well. We prove that $H$ admits at least $\Omega\left(p^{2\binom{r}{2}}n^{2r}\right)$ $\psi$-rainbow copies of $K_{2r}$. To this end, let $\R_W \subseteq \C_W$ be the collection of members of $\C_W$ inducing a $\psi$-rainbow copy of $K_r$, and let $\R_U \subseteq \binom{U}{r}$ be the collection of $\psi$-rainbow copies of $K_r$ supported on $H[U]$. Define an auxiliary bipartite graph $\F$ whose vertex bipartition is given by $(\R_W,\R_U)$ with $X \in \R_W$ and $Y \in \R_U$ forming an edge of $\F$ if and only if
		\begin{description}
			\item [(i)] 
				$Y \in \B_{X,\psi}$, that is, $H[X \cup Y] \cong K_{2r}$ and $H[X] \cup H[X,Y]$ is $\psi$-rainbow; and 
			\item [(ii)]
				$\psi(E_{H}(Y)) \cap \psi(E_{H}(X)) = \emptyset$, i.e., there is no {\sl clash} between the set of colours seen on $H[X]$ and that seen on $H[Y]$.
		\end{description}
		
		A pair $\{X,Y\}$, as above, forming an edge of $\F$, forms a copy of $K_{2r}$ in $H$ such that $H[X] \cup H[X,Y]$ is $\psi$-rainbow and $H[X] \cup H[Y]$ is $\psi$-rainbow. Still, such a copy of $K_{2r}$ may not be $\psi$-rainbow as a clash between the colours seen on $H[Y]$ and $H[X,Y]$ is still possible; such clashes are dealt with below.  

		Combined with the following claim, Properties~\ref{itm:few-cliques-total}, \ref{itm:many-cliques-CW} and \ref{itm:few-cliques-U} imply that $|\R_W|, |\R_U| = \Theta\left(p^{\binom{r}{2}}n^r\right)$ and that $\F$ is {\sl dense}. 

		\begin{claim}\label{eq:F-dense}
			For every $X \in \R_W$, $\deg_{\F}(X) = \Omega\left(p^{\binom{r}{2}}n^r\right)$ holds.
		\end{claim}

		\begin{proof}
			Fix $X \in \R_W$ and let $\Gamma_{X,\psi} \subseteq H[N_X]$ be the subgraph whose existence is guaranteed by Property~\ref{itm:gamma-exists}. For a colour $c \in \psi(E_{H}(X))$ seen on some edge of $X$, let $M_c \subseteq E(\Gamma_{X,\psi})$ be the matching in $\Gamma_{X,\psi}$ induced by the colour $c$. Standard regularity arguments, assert that the graph 
			$$
			\Gamma'_{X,\psi} := \Gamma_{X,\psi} \setminus \bigcup_{c \in \psi(E_{H}(X))} M_c,
			$$ 
			obtained from $\Gamma_{X,\psi}$ by removing all edges coloured using a colour seen on the edges of $H[X]$, is a member of the graph family specified in~\eqref{eq:KLR-counting-prop}. Indeed, 
			$
			|\cup_{c \in \psi(E_{\Gamma}(X))} M_c| = O_r(n) = o(pn^2)
			$,
			where the last equality holds since $p = \omega(n^{-1})$. Consequently, the inter-cluster density of the K{\L}R-graph $\Gamma_{X,\psi}$ as well as its regularity parameter are worsened by a factor of at most $2$, say. 

			By Property~\ref{itm:gamma-exists}, every copy of $K_r$ in $\Gamma_{X,\psi}$ is $\psi$-rainbow. At least $\alpha p^{\binom{r}{2}}\binom{|N_X|}{r}$ of these $\psi$-rainbow copies of $K_r$ are retained in $\Gamma'_{X,\psi}$ by Property~\ref{itm:many-cliques-gamma}. Since $|N_X| = \Omega_{d', \eps}(n)$ and, by Property~\ref{itm:many-cliques-BW}, all but at most $\frac{\alpha}{2} p^{\binom{r}{2}} \binom{|N_X|}{r}$ 
			copies of $K_r$ in $H[N_X]$ are supported on the members of $\B_{X,\psi}$, the claim follows.
		\end{proof}

		It remains to prove that there are $\Omega\left(p^{2\binom{r}{2}}n^{2r}\right)$ edges $\{X,Y\}$ of $\F$ with $X \in \R_W$ and $Y \in \R_U$ such that $H[Y]$ and $H[X,Y]$ do not clash, as each such pair gives rise to a distinct $\psi$-rainbow copy of $K_{2r}$. Let $\F'$ be the spanning subgraph of $\F$ whose edges are pairs $\{X, Y\} \in E(\F)$ such that $H[Y]$ and $H[X, Y]$ do clash.
		We claim that 
		\begin{equation}\label{eq:k_Y}
			\deg_{\F'}(Y) = o\left(p^{\binom{r}{2}} n^r\right)
		\end{equation}
		for every $Y \in \R_U$. To prove~\eqref{eq:k_Y}, fix $Y \in \R_U$ and let $S_Y \subseteq W$ be the set of vertices $s \in W$ such that $H[Y]$ and $H[\{s\}, Y]$ clash. As $\psi$ is proper, each edge $e$ in $H[Y]$ gives rise to at most $r-2$ edges from $Y$ to $W$ of colour $\psi(e)$. It follows that $|S_Y| \le (r-2)\binom{r}{2}$. Note that every $X \in \R_W$ such that $\{X, Y\} \in E(\F')$ intersects $S_Y$. Property~\ref{itm:few-cliques-vertex} then implies that 
		$\deg_{\F'}(Y) = o\left(|S_Y| p^{\binom{r}{2}} n^r\right) = o\left(p^{\binom{r}{2}} n^r\right)$ and~\eqref{eq:k_Y} follows. Since $|\R_U| = \Theta\left(p^{\binom{r}{2}} n^{r}\right)$, we conclude that $e(\F') = o\left(p^{2\binom{r}{2}}n^{2r}\right)$. Since $e(\F) = \Omega\left(p^{2\binom{r}{2}}n^{2r}\right)$ holds by Claim~\ref{eq:F-dense} and by the established fact that $|\R_W| = \Theta\left(p^{\binom{r}{2}} n^{r}\right)$, we deduce that
		$$
			|E(\F \setminus \F')| = \Omega\left(p^{2\binom{r}{2}}n^{2r}\right),
		$$
		i.e.,\ the number of pairs $\{X,Y\} \in E(\F)$ giving rise to a $\psi$-rainbow $K_{2r}$ is $\Omega\left(p^{2\binom{r}{2}}n^{2r}\right)$.
		This concludes our proof of Theorem~\ref{thm:main}.
	\end{proof}
\section{Rainbow odd cycles} \label{sec:odd-cylces}

	In this section we prove Proposition~\ref{thm:main:odd-cycles}.

	\begin{proof}[Proof of Proposition~\ref{thm:main:odd-cycles}]
		To see the $0$-statement, fix some $\ell \geq 2$ and $d \leq 1/2$, and let $G$ be a bipartite graph on $n$ vertices with edge-density at least $d$. Since $G$ is bipartite, any copy of $C_{2\ell-1}$ in $\Gamma \sim G \cup \Gnp$ must contain some edge of $\Gnp$. However, $\Gnp$ is a.a.s.\ empty whenever $p = o \left(n^{-2} \right)$. In particular, a.a.s.\ no edge-colouring of $\Gamma$ can yield a rainbow $C_{2\ell-1}$ for any $\ell \geq 2$.

		Proceeding to the $1$-statement, let $\ell \geq 2$ and $d > 0$ be fixed, and let $G \in \sG_{d,n}$ be given. 
		Apply Lemma~\ref{lem:dependent-random-choice} with parameters $a := d^{2\ell} \cdot n - 1$, $\overline{d} := dn, r := \ell$, and $m := \sqrt{n}$. Note that
		\[
			\binom{n}{r} \left(\frac{m}{n}\right)^t 
			\le n^{\ell} \cdot n^{-2\ell/2} = 1
		\]
		holds for $t = 2\ell$. In particular, the inequality in Lemma~\ref{lem:dependent-random-choice} holds for this choice of parameters. It follows that there exists a set $U$ of at least $d^{2\ell} \cdot n - 1$ vertices such that every $\ell$ vertices in $U$ have at least $\sqrt{n}$ common neighbours. The probability that there are no edges of $\Gnp$ in $U$ is 
		\[
			(1 - p)^{\binom{|U|}{2}} \le e^{-\omega(n^{-2} \cdot n^2)} = o(1).
		\]
		In other words, a.a.s.\ there is an edge of $\Gnp$ with both ends in $U$; denote such an edge by $xy$, and let $X$ be a subset of $U$ of size $\ell$ that contains $x$ and $y$. By the assumption on $U$, the set $Z$ of common neighbours of $X$ has size at least $\sqrt{n}$. We claim that $G[X, Z] \cup \{xy\}$ contains a rainbow $C_{2\ell-1}$ for every proper colouring $\psi$. 

		Let $A = \{z \in Z : \exists w \in X \textrm { such that } \psi(zw) = \psi(xy)\}$. Note that $|A| \leq \ell - 2$ as $\psi$ is a proper edge-colouring. Let $Z' = Z \setminus A$. Now recall that the number of non-rainbow copies of $K_{\ell,\ell-1}$ in $G[X, Z']$, with $X$ being their $\ell$-part is at most $O(|Z'|^{\ell-2})$ (see Observation~\ref{obs:non-rainbow-Krr}). As there are $\binom{|Z'|}{\ell-1}$ such copies of $K_{\ell,\ell-1}$ in $G[X, Z']$, it follows that there is a rainbow copy of $K_{\ell,\ell-1}$ in $G[X, Z']$, whose vertex-set is $X \cup Z''$, where $Z''$ is a subset of $Z'$ of size $\ell-1$. Consider the graph $H$ with vertices $X \cup Z''$ and edge-set $(X \times Z'') \cup \{xy\}$. It is a rainbow subgraph of $G$ (by choice of $Z'$ and $Z''$) which clearly contains a copy of $C_{2\ell-1}$. This concludes the proof of the proposition.
	\end{proof}

\subsection*{Acknowledgements}
	We would like to thank the anonymous referees for their many insightful and helpful comments.

\bibliographystyle{amsplain}
\bibliography{lit}

\end{document}